\documentclass[11pt,a4paper]{article}
\usepackage[latin1]{inputenc}
\usepackage{amsmath}
\usepackage{amsthm}
\usepackage{amsfonts}
\usepackage{amsfonts,amsthm,latexsym,amsmath,amssymb,amscd,epsfig,psfrag,enumerate}
\usepackage{graphics,graphicx, bezier, float, color, hyperref}
\usepackage{amssymb,url}
\usepackage{multienum}
\usepackage[table]{xcolor}
\usepackage{multicol,multirow}
\usepackage{graphicx}
\usepackage{fancyvrb}
\usepackage{tcolorbox}
\usepackage{parskip}
\usepackage[toc,page]{appendix}
\usepackage[top=2.7cm, bottom=2.7cm, left=1.5cm, right=1.5cm]{geometry}
\usepackage{xcolor}
\usepackage[none]{hyphenat}[section]
\usepackage{blkarray}
\newtheorem{theorem}{Theorem}[section]
\newtheorem{lemma}[theorem]{Lemma}

\newtheorem{remark}[theorem]{Remark}
\newtheorem{definition}[theorem]{Definition}

\numberwithin{equation}{section}
\numberwithin{table}{section}
\numberwithin{figure}{section}

\title{On generalized Thabit numbers $(p+1)p^\mathfrak{a}-1$ in the $k$-Lucas sequence}
\author{Herbert Batte$^{1,*} $, Florian Luca$^{1,2}$ and Pantelimon St\u anic\u a$^{3}$}
\date{}

\begin{document}
\maketitle
\abstract{ Let $k\ge 2$ and $\{L_n^{(k)}\}_{n\geq 2-k}$ be the sequence of $k$-Lucas numbers whose first $k$ terms are $0,\ldots,0,2,1$ and each term afterwards is the sum of the preceding $k$ terms. In this paper, we solve the Diophantine equation $L_n^{(k)}=(p+1)p^\mathfrak{a}-1$, for a Mersenne or Fermat prime $p=2^{\ell}\pm 1$, and positive integers $n\ge 2$, $k\ge 2$, $\mathfrak{a}\ge 1$ and $\ell \ge 1$.} 

\medskip


\medskip


\medskip

\thanks{$ ^{*} $ Corresponding author}

\section{Introduction}\label{intro}
\subsection{Background}
Recurrence sequences occupy a central place in modern Number theory, with their properties and Diophantine aspects attracting sustained research interest. The Fibonacci sequence and its companion Lucas sequence are among the most celebrated examples, and both admit natural higher-order generalizations. For an integer
$k \geq 2$, the $k$-Fibonacci sequence $\{F_n^{(k)}\}_{n\in\mathbb{Z}}$ is defined by the linear recurrence
\begin{equation*}
	F_n^{(k)} = F_{n-1}^{(k)} + F_{n-2}^{(k)} + \cdots + F_{n-k}^{(k)},
	\quad \text{for all} \ n \ge 2,
\end{equation*}
subject to the initial conditions $F_{2-k}^{(k)} = \cdots = F_{-1}^{(k)} = F_{0}^{(k)} = 0$ and $F_1^{(k)} = 1$. Setting $k = 2$ recovers the classical Fibonacci sequence, and the family $\{F_n^{(k)}\}$ has attracted considerable attention from algebraic and combinatorial perspectives alike.
The companion sequence of $k$-Lucas numbers $\{L_n^{(k)}\}_{n\in\mathbb{Z}}$ satisfies the same recurrence,
\begin{equation*}
	L_n^{(k)} = L_{n-1}^{(k)} + L_{n-2}^{(k)} + \cdots + L_{n-k}^{(k)},
	\quad \text{for all} \ n \ge 2,
\end{equation*}
but is distinguished by the initial values $L_{2-k}^{(k)} = \cdots = L_{-1}^{(k)} = 0,~L_0^{(k)}=2,~L_1^{(k)}=1$. The case $k=2$ yields the classical Lucas sequence, which has found applications across Number theory, cryptography, and primality testing. Both $k$- families exhibit a rich arithmetic structure that continues to motivate new investigations.

For a prime $p$ and an integer $\mathfrak{a}\ge 1$, the expression 
\begin{align*}
	(p+1)p^\mathfrak{a}-1,
\end{align*}
is a natural prime-base analogue of the classical Thabit-type number $3\cdot 2^\mathfrak{a}-1$. Indeed, when $p=2$, we recover the Thabit numbers $3\cdot 2^\mathfrak{a}-1$, for an integer $\mathfrak{a}\ge 1$. The identity 
\begin{align*}
	(p+1)p^\mathfrak{a}=\left[(p+1)p^\mathfrak{a}-1\right]+1,
\end{align*}
shows that one is searching for terms in a linear recurrence for which the successor has a rigid multiplicative structure ``\texttt{a prime power times its successor}". This arithmetic structure is the basic input throughout the paper.

The arithmetic study of special Number-theoretic sequences inside linear sequences has a long history. For example, Altassan and Alan proved in \cite{alta} that the only Mersenne numbers $2^\mathfrak{a}-1$, that intersect with the sequence of $k$-Lucas numbers are $1=L_1^{(k)}$, $3=L_2^{(k)}$ and $7=L_4^{(2)}$. 

In this paper, we study the Diophantine equation 
\begin{align}\label{eq:main}
	L_n^{(k)} = (p+1)p^\mathfrak{a}-1,
\end{align}
where $p=2^{\ell}\pm 1$ is a Mersenne or Fermat prime, and $n\ge 2$, $k\ge 2$, $\mathfrak{a}\ge 1$, $\ell \ge 1$ are positive integers. Note that for $n=0,1$, we have $L_n^{(k)}=2,1$ and the Diophantine equation \eqref{eq:main} has no positive solutions $p$ and $\mathfrak{a}$.

\subsection{Main Result}
We prove the following result.

\begin{theorem}
\label{thm1.1} 
Let $p=2^{\ell}\pm 1$ be a Mersenne or Fermat prime and $n\ge 2$, $k\ge 2$, $\mathfrak{a}\ge 1$, $\ell \ge 1$ be positive integers.
Then, the only solutions to the Diophantine equation 
\begin{align*}
	L_n^{(k)} = (p+1)p^\mathfrak{a}-1,
\end{align*}
are 
\begin{align*}
	(n,k,p,\mathfrak{a},\ell) \in \{(5,2,3,1,2),(7,2,5,1,2),(6,3,3,2,2)\}.
\end{align*}
\end{theorem}
These solutions in Theorem \ref{thm1.1} are summarized in Table \ref{tab:solutions} below.

\begin{table}[H]
	\centering
	\renewcommand{\arraystretch}{1.3}
	\begin{tabular}{ccccccc}
		\hline
		$n$ & $k$ & $p$ & $\mathfrak{a}$ & $\ell$ & Type & $L_n^{(k)}$ \\
		\hline
		$5$ & $2$ & $3=2^2-1=2^1+1$ & $1$ & $2$ & Mersenne/Fermat & $11$ \\
		$6$ & $3$ & $3=2^2-1=2^1+1$ & $2$ & $2$ & Mersenne/Fermat & $35$ \\
		$7$ & $2$ & $5=2^2+1$        & $1$ & $2$ & Fermat          & $29$ \\
		\hline
	\end{tabular}
	\caption{All solutions $(n,k,p,\mathfrak{a},\ell)$ to \eqref{eq:main}
		with $p=2^{\ell}\pm 1$ a Mersenne or Fermat prime.}
	\label{tab:solutions}
\end{table}
\section{Some useful formulas and inequalities involving $L_n^{(k)}$}

In this section we collect some useful formulas and inequalities concerning $k$-generalized Lucas numbers. First, it is known that the formula
\begin{align*}
	L_n^{(k)} = 3 \cdot 2^{n-2}\quad {\text{\rm holds~for}}\quad 2\le n\le k.
\end{align*}
Furthermore, 
\[
L_{k+1}^{(k)} = 3 \cdot 2^{k-1} - 2.
\]
Using induction, one proves easily that the inequality
\begin{equation}
\label{eq:3times2}
	L_n^{(k)} \le 3 \cdot 2^{n-2}-2\qquad {\text{\rm holds~for~all}}\qquad n\ge k+1.
\end{equation}
The $k$-generalized Lucas numbers can be expressed in terms of the $k$-generalized Fibonacci numbers via the formula
\begin{equation}
\label{eq:FL}
L_n^{(k)}=2F_{n+1}^{(k)}-F_n^{(k)}.
\end{equation}
To prove the above formula, it is enough to first notice that it holds for $n=-(k-2),\ldots,0,1,2$ and then to invoke the fact that both sequences $\{L_n^{(k)}\}_{n\in {\mathbb Z}}$ and $\{2F_{n+1}^{(k)}-F_n^{(k)}\}_{n\in {\mathbb Z}}$ 
satisfy the same $k$th order linear recurrence to conclude that the above formula holds for all $n\in {\mathbb Z}$. Another useful formula is 
$$
L_{n}^{(k)}=2L_{n-1}^{(k)}-L_{n-(k+1)}^{(k)}
$$
for all $n\in {\mathbb Z}$ (to see why it holds, replace the left-hand side above by $L_{n-1}^{(k)}+L_{n-2}^{(k)}+\cdots+L_{n-k}^{(k)}$, cancel a term $L_{n-1}^{(k)}$ on both sides of the resulting equation, and recognise that 
the resulting formula is just the formula representing $L_{n-1}^{(k)}$ as the sum of the previous $k$ terms of the sequence).  Reducing the above formula modulo $2$ we get
$$
L_n^{(k)}\equiv L_{n-(k+1)}^{(k)} \pmod 2,
$$
showing that $\{L_n^{(k)}\}_{n\in {\mathbb Z}}$ is periodic modulo $2$ with period $k+1$. This fact is so useful that we record it as a lemma.

\begin{lemma}
\label{lem:1}
The sequence $\{L_n^{(k)}\}_{n\in {\mathbb Z}}$ is periodic modulo $2$ with period $k+1$.
\end{lemma} 

Now let us recall some more formulas related to the sequence of $k$-generalized Lucas numbers. Its characteristic  polynomial
$$
g_k(X) = X^k - X^{k-1} - \cdots - X - 1,
$$
is irreducible in $\mathbb{Q}[X]$. This polynomial has exactly one real root greater than 1, denoted by $\alpha := \alpha(k)$, and all its other roots lie inside the unit circle in the complex plane. Further, this real root lies within the range
\begin{align}\label{eq2.3}
	2\left(1 - 2^{-k} \right) < \alpha < 2,\qquad {\text{\rm for~all}}\qquad k\ge 2.
\end{align}
If we need to refer to all the roots of $g_k(X)$, then we label them as $\alpha_1,\ldots,\alpha_k$ with the convention that $\alpha:=\alpha_1$. For all $n \geq 0$ and $k \geq 2$, it was proved in \cite{BRL} that the $k$-generalized Lucas numbers are bounded by powers of $\alpha$ as
\begin{align}\label{eq2.4}
	\alpha^{n-1} \leq L_n^{(k)} \leq 2\alpha^n.
\end{align}
Moreover, if we compare \eqref{eq:main} and \eqref{eq2.4}, we deduce that
\begin{align*}
	p^2 \le p^{\mathfrak{a}+1}=p\cdot  p^\mathfrak{a}<(p+1)p^\mathfrak{a}=1+ L_n^{(k)} \leq 3\alpha^n,
\end{align*}
which implies that
\begin{align}\label{a-bound}
	p < \sqrt3 \alpha^{n/2} \qquad \text{and}\qquad \mathfrak{a} \le n.
\end{align}

Next, we define the function
\begin{equation*}
	f_k(x) := \frac{x - 1}{2 + (k + 1)(x - 2)},\qquad x\in {\mathbb C}\backslash \left\{2-{2}/{(k+1)}\right\}.
\end{equation*}
We have 
$$
\frac{df_k(x)}{dx}=-\frac{(k-1)}{(2+(k+1)(x-2))^2}<0,\qquad {\text{\rm for all}}\qquad x>0.
$$
In particular, inequality \eqref{eq2.3} implies that
\begin{align}\label{eq2.5}
	\dfrac{1}{2}=f_k(2)<f_k(\alpha)<f_k(2(1 - 2^{-k} ))\le \dfrac{3}{4},
\end{align}
for all $k\ge 3$. It is easy to check that the above inequality holds for $k=2$ as well. Further, it is easy to verify that $|f_k(\alpha_i)|<1$, for all $2\le i\le k$, where $\alpha_i$ are the remaining 
roots of $\Psi_k(x)$ for $i=2,\ldots,k$.

The Binet formula for the general term of the $k$-Lucas numbers is given by
\begin{equation*}
L_n^{(k)} = \sum_{i=1}^k (2\alpha_i - 1)f_k(\alpha_i)\alpha_i^{n - 1},\qquad {\text{\rm for all}}\quad n\in {\mathbb Z}.
\end{equation*}
Since $\alpha_2,\ldots,\alpha_k$ are inside the unit circle, one would say that only the first term $(2\alpha_1-1)f_k(\alpha_1)\alpha_1^{n-1}$ already approximates $L_n^{(k)}$ and indeed the approximation
\begin{align}\label{lk_b}
\left|L_n^{(k)} - f_k(\alpha)(2\alpha - 1)\alpha^{n - 1}\right| < \frac{3}{2},\qquad {\text{\rm for~all}}\quad n\ge 2-k
\end{align}
appears, for example, in \cite{BRL}. This tells us that most of the size of $L_n^{(k)}$ comes from the dominant term involving the real root $\alpha$, while the other terms contribute very little. A better estimate than \eqref{lk_b} appears in Section  3.3 page 14 of \cite{bat}, but with a more restricted range of $n$ in terms of $k$. It states that 
\begin{align}\label{lk_b1}
	\left| f_k(\alpha)(2\alpha - 1)\alpha^{n - 1}-3\cdot 2^{n-2}\right| < 3\cdot 2^{n-2}\cdot \frac{36}{2^{k/2}}\qquad {\text{\rm provided}}\qquad n<2^{k/2}.
\end{align}
Another useful combinatorial formula for the $k$th generalized Fibonacci numbers is  due to Cooper and Howard in \cite{cop} and states that
\begin{align}\label{comb}
F_n^{(k)} = 2^{n - 2} + \sum_{j = 1}^{\left\lfloor \frac{n + k}{k + 1} \right\rfloor - 1} C_{n, j} \, 2^{n - j(k+1) - 2}, \quad \text{where} \quad 
C_{n,j} := (-1)^j \left( \binom{n - jk}{j} - \binom{n - jk - 2}{j - 2} \right).
\end{align}
In Equation \eqref{comb}, we make the convention that $\displaystyle\binom{n}{m}=0$ if either $n<m$ or one of $n$ or $m$ is negative.
Via equation \eqref{eq:FL}, it yields the following formula for the $k$th generalized Lucas number:
 \begin{align}\label{2adic2}
	L_n^{(k)} &= 2F_{n+1}^{(k)}-F_n^{(k)}\nonumber\\
	&= 2\left(2^{n-1}+
	\sum_{j = 1}^{\left\lfloor \frac{n +1+ k}{k + 1} \right\rfloor - 1} C_{n+1, j} \, 2^{n - j(k+1) - 1}\right)-
	2^{n - 2} - \sum_{j = 1}^{\left\lfloor \frac{n + k}{k + 1} \right\rfloor - 1} C_{n, j} \, 2^{n - j(k+1) - 2}\nonumber\\
	&= 3\cdot2^{n-2}+
	\sum_{j = 1}^{\left\lfloor \frac{n +1+ k}{k + 1} \right\rfloor - 1}4 C_{n+1, j} \, 2^{n - j(k+1) -2}
	- \sum_{j = 1}^{\left\lfloor \frac{n + k}{k + 1} \right\rfloor - 1} C_{n, j} \, 2^{n - j(k+1) - 2}.
\end{align}

Lastly here, we recall one additional simple fact from calculus. If $x\in \mathbb{R}$ satisfies $|x|<1/2$, then 
\begin{align}\label{eq2.5g}
	|\log(1+x)|&<|x-x^2/2+-\dots|
	<|x|+\frac{|x|^2+|x|^3+\dots}2
	<|x|\left(1+\frac{|x|}{2(1-|x|)}\right)<\frac 32 |x|.
\end{align}
In a similar way, we obtain the lower bound $|\log(1+x)|>\frac 12 |x|$ provided that $|x|<1/2$. We shall use these inequalities later in the proof of the main result.

\section{The $2$-adic valuation of $L_n^{(k)}$}

Since $\{L_n^{(k)}\}_{n\in {\mathbb Z}}$ is periodic modulo $2$ with period $k+1$ (see Lemma \ref{lem:1}), it follows that the $2$-adic valuation of $L_n^{(k)}$ depends on the residue of $n$ modulo $k+1$. So, let us write 
$n=r+m(k+1)$ for some integers $m\ge 0$ and $r\in \{0,1,\ldots,k\}$. 

\begin{lemma}
\label{lem:2adic}
Let $k\ge 2$ and  $n=r+m(k+1)$ with integers $m\ge 0$ and $r\in \{0,1,\ldots,k\}$. The following congruences hold.
\begin{enumerate}[\upshape(i)]
\item If $r=0$, then 
\begin{equation}
\label{eq:r=0}
L_n^{(k)}\equiv 2(-1)^m \pmod {2^{k-2}}.
\end{equation}
\item If $r=1$, then 
\begin{equation}
\label{eq:r=1}
L_n^{(k)}\equiv (4m+1)(-1)^m\pmod {2^{k-1}}.
\end{equation}
\item If $r=2$, then 
\begin{equation}
\label{eq:r=2}
L_n^{(k)}\equiv (4m^2+6m+3)(-1)^m\pmod {2^{k}}.
\end{equation}
\item If $r\ge 3$, then 
\begin{align}
\label{eq:rge3}
L_n^{(k)}\equiv &(-1)^m 2^{r-2} \left(4\left(\binom{m+r+1}{m}-\binom{m+r-1}{m-2}\right)-\left(\binom{m+r}{m}-\binom{m+r-2}{m-2}\right)\right)\nonumber\\
&\hspace*{10cm}\pmod {2^{k+r-2}}.
\end{align}
\end{enumerate}
\end{lemma}

\begin{proof}
Note that 
$$
\left\lfloor \frac{n+1+k}{k+1}\right\rfloor=\left\lfloor \frac{r+(m+1)(k+1)}{k+1}\right\rfloor=m+1,
$$
while
$$
\left\lfloor \frac{n+k}{k+1}\right\rfloor=\left\lfloor \frac{(r+k)+m(k+1)}{k+1}\right\rfloor=\left\{\begin{matrix} m+1 & {\text{\rm if}} & r\ge 1,\\ m & {\text{\rm if}} & r=0. \end{matrix}\right.
$$
Furthermore,  for $j$ in the summation ranges in \eqref{2adic2}, we have 
$$
n-j(k+1)-2=(r-2)+(m-j)(k+1)\ge k+r-2\qquad {\text{\rm except~when}}\qquad j=m.
$$
The case $j=m$ happens in the last term of the first summation in \eqref{2adic2} and also in the last term of the second summation but only when $r\ge 1$. So, when $r=0$, assuming also that 
$m\ge 1$ (otherwise $n=0$ so $L_n^{(k)}=L_0^{(k)}=2$ and congruence \eqref{eq:r=0} holds), we have that 
\begin{eqnarray*}
L_n^{(k)} & \equiv &  2^{0-2}\cdot 4 C_{n+1,m} \pmod {2^{k+0-2}}\\
&\equiv& (-1)^m \left(\binom{n+1-mk}{m}-\binom{n+1-mk-2}{m-2}\right)\pmod {2^{k-2}}\\
& \equiv &  (-1)^m \left(\binom{1+m(k+1)-mk}{m}-\binom{1+m(k+1)-mk-2}{m-2}\right)\pmod {2^{k-2}}\\
& \equiv &  (-1)^m \left(\binom{m+1}{m}-\binom{m-1}{m-2}\right) \equiv  2(-1)^m \pmod {2^{k-2}}.
\end{eqnarray*}
This proves (i). From now on, $r\ge 1$ so both sums appearing in \eqref{2adic2} have the same length $m$. It follows that 
$$
L_n^{(k)}\equiv 2^{r-2}\left(4C_{n+1,m}-C_{n,m}\right)\pmod {2^{k+r-2}}.
$$
We treat the cases $r=1,~r=2$ and $r\ge 3$ separately. For $r=1$, if $m=0$, then $L_n^{(k)}=L_1^{(k)}=1$ and congruence \eqref{eq:r=1} holds. For $m\ge 1$, we get
\begin{eqnarray*}
L_n^{(k)} & \equiv & 2(-1)^m \left(\binom{n+1-mk}{m}-\binom{n-1-mk}{m}\right)-2^{-1}(-1)^m \left(\binom{n-mk}{m}-\binom{n-2-mk}{m-2}\right)\\
&~&\hfill\pmod {2^{k-1}}\\
& \equiv &  
2(-1)^m \left(\binom{2+m(k+1)-mk}{m}-\binom{m(k+1)-mk}{m-2}\right)\\ 
& ~& \qquad-2^{-1}(-1)^m \left(\binom{1+m(k+1)-mk}{m}-\binom{-1+m(k+1)-mk}{m-2}\right)\pmod {2^{k-1}}\\
& \equiv &  
2(-1)^m \left(\binom{m+2}{m}-\binom{m}{m-2}\right)-2^{-1}(-1)^m\left(\binom{m+1}{m}-\binom{m-1}{m-2}\right)\pmod {2^{k-1}}.\\
\end{eqnarray*}
When $m=1$, we get 
$$
-2\binom{3}{1}+2^{-1}\binom{2}{1}\equiv -5\pmod {2^{k-1}}\equiv (-1)^1(4\cdot 1+1)\pmod {2^{k-1}},
$$
while when $m\ge 2$ we get
\begin{eqnarray*}
L_n^{(k)} & \equiv &  2(-1)^m \left(\binom{m+2}{2}-\binom{m}{2}\right)-2^{-1}(-1)^m\left(\binom{m+1}{1}-\binom{m-1}{1}\right)\pmod {2^{k-1}}\\
& \equiv &   2(-1)^m \left(\frac{(m+2)(m+1)}{2}-\frac{m(m-1)}{2}\right)-(-1)^m\pmod {2^{k-1}}\equiv (-1)^m(4m+1)\\
&~&\hfill\pmod {2^{k-1}}.
\end{eqnarray*}
This takes care of the congruence \eqref{eq:r=1} in (ii). For (iii), let $r=2$. When $m=0$, we have $n=2$, so $L_n^{(k)}=3=(-1)^0(4\cdot 0^2+6\cdot 0+3)$. When $m\ge 1$, we get
\begin{eqnarray*}
L_n^{(k)} & \equiv & 2^2(-1)^m \left(\binom{n+1-mk}{m}-\binom{n-1-mk}{m}\right)-2^{-1}(-1)^m \left(\binom{n-mk}{m}-\binom{n-2-mk}{m-2}\right)\\
&~&\hfill\pmod {2^{k}}\\
& \equiv &  
2(-1)^m \left(\binom{3+m(k+1)-mk}{m}-\binom{1+m(k+1)-mk}{m-2}\right)\\ 
& -& (-1)^m \left(\binom{2+m(k+1)-mk}{m}-\binom{m(k+1)-mk}{m-2}\right)\pmod {2^{k}}\\
& \equiv &  
4(-1)^m \left(\binom{m+3}{m}-\binom{m+1}{m-2}\right)-(-1)^m\left(\binom{m+2}{m}-\binom{m}{m-2}\right)\pmod {2^{k}}.\\
\end{eqnarray*}
When $m=1$, we get 
$$
-4\binom{4}{1}+\binom{3}{1}\equiv -13\pmod {2^{k-1}}\equiv (-1)^1(4\cdot 1^2+6\cdot 1+3)\pmod {2^{k}},
$$
while when $m\ge 2$ we get
\begin{eqnarray*}
L_n^{(k)} & \equiv &  4(-1)^m \left(\binom{m+3}{3}-\binom{m+1}{3}\right)-(-1)^m\left(\binom{m+2}{2}-\binom{m}{2}\right)\pmod {2^{k}}\\
& \equiv &  
(-1)^m(4m^2+6m+3)\pmod {2^{k}},
\end{eqnarray*}
which proves \eqref{eq:r=2}.
Part (iv) is similar, that is, let $r\ge 3$ and $m\ge 0$. Then \eqref{2adic2} gives
\begin{eqnarray*}
	L_n^{(k)} & = & 3\cdot2^{n-2}+
	\sum_{j = 1}^{m}4 C_{n+1, j} \, 2^{n - j(k+1) -2}
	- \sum_{j = 1}^{m} C_{n, j} \, 2^{n - j(k+1) - 2}\\
	& = & 3\cdot2^{n-2}+
	\sum_{j = 1}^{m}\left(4 C_{n+1, j}-C_{n, j}\right) \, 2^{n - j(k+1) -2}\\
	&\equiv & (-1)^m 2^{r-2} \left(4\left(\binom{m+r+1}{m}-\binom{m+r-1}{m-2}\right)-\left(\binom{m+r}{m}-\binom{m+r-2}{m-2}\right)\right)\\
	&~&\hfill\pmod {2^{k+r-2}},
\end{eqnarray*}
which is \eqref{eq:rge3}.
\end{proof}

\section{Methods}
\subsection{Linear forms in logarithms}
To estimate how small a nonzero linear combination of logarithms of algebraic numbers can be, we apply a result known as a Baker-type lower bound. While several such bounds are known in the literature, we make use of a version due to Matveev, as presented in \cite{matl}. Before we can formulate such inequalities we need the notion of height of an algebraic number recalled below.

\begin{definition}
	Let $ \gamma $ be an algebraic number of degree $ d $ with minimal primitive polynomial over the integers $$ a_{0}x^{d}+a_{1}x^{d-1}+\cdots+a_{d}=a_{0}\prod_{i=1}^{d}(x-\gamma^{(i)}), $$ where the leading coefficient $ a_{0} $ is positive. Then, the logarithmic height of $ \gamma$ is given by $$ h(\gamma):= \dfrac{1}{d}\Big(\log a_{0}+\sum_{i=1}^{d}\log \max\{|\gamma^{(i)}|,1\} \Big). $$
\end{definition}
In particular, if $ \gamma$ is a rational number represented as $\gamma:=p/q$ with coprime integers $p$ and $ q\ge 1$, then $ h(\gamma ) = \log \max\{|p|, q\} $. 
The following properties of the logarithmic height function $ h(\cdot) $ will be used in the rest of the paper without further reference:
\begin{equation}\nonumber
	\begin{aligned}
		h(\gamma_{1}\pm\gamma_{2}) &\leq h(\gamma_{1})+h(\gamma_{2})+\log 2;\\
		h(\gamma_{1}\gamma_{2}^{\pm 1} ) &\leq h(\gamma_{1})+h(\gamma_{2});\\
		h(\gamma^{s}) &= |s|h(\gamma)  \quad {\text{\rm valid for}}\quad s\in \mathbb{Z}.
	\end{aligned}
\end{equation}
With these properties, it was easily computed in Section 3 of \cite{Brl} that
\begin{align}\label{eqh}
	h\left(f_k(\alpha)\right)<3\log k, \qquad \text{for all}\qquad k\ge 2.
\end{align}

A linear form in logarithms is an expression
\begin{equation*}
	\Lambda:=b_1\log \gamma_1+\cdots+b_t\log \gamma_t,
\end{equation*}
where for us $\gamma_1,\ldots,\gamma_t$ are positive real  algebraic numbers and $b_1,\ldots,b_t$ are nonzero integers. We assume, $\Lambda\ne 0$. We need lower bounds 
for $|\Lambda|$. We write ${\mathbb K}:={\mathbb Q}(\gamma_1,\ldots,\gamma_t)$ and $D$ for the degree of ${\mathbb K}$.
We start with the general form due to Matveev, see Theorem 9.4 in \cite{matl}. 

\begin{theorem}[Matveev, see Theorem 9.4 in \cite{matl}]
	\label{thm:Matl} 
	Put $\Gamma:=\gamma_1^{b_1}\cdots \gamma_t^{b_t}-1=e^{\Lambda}-1$. Assume $\Gamma\ne 0$. Then 
	$$
	\log |\Gamma|>-1.4\cdot 30^{t+3}\cdot t^{4.5} \cdot D^2 (1+\log D)(1+\log B)A_1\cdots A_t,
	$$
	where $B\ge \max\{|b_1|,\ldots,|b_t|\}$ and $A_i\ge \max\{Dh(\gamma_i),|\log \gamma_i|,0.16\}$ for $i=1,\ldots,t$.
\end{theorem}

\subsection{Reduction methods}
Typically, the estimates from Matveev's theorem are excessively large to be practical in computations. To refine these estimates, we use a method based on the LLL-algorithm. While the Baker-Davenport reduction is often sufficient for linear forms in two or three logarithms, the LLL-algorithm offers a more robust and generalized approach for three or more logarithms, as is the case in our study (see \cite{SMA, Weg}). We next explain this method.

Let $k$ be a positive integer. A subset $\mathcal{L}$ of the $k$-dimensional real vector space ${ \mathbb{R}^k}$ is called a lattice if there exists a basis $\{b_1, b_2, \ldots, b_k \}$ of $\mathbb{R}^k$ such that
\begin{align*}
	\mathcal{L} = \sum_{i=1}^{k} \mathbb{Z} b_i = \left\{ \sum_{i=1}^{k} r_i b_i \mid r_i \in \mathbb{Z} \right\}.
\end{align*}
We say that $b_1, b_2, \ldots, b_k$ form a basis for $\mathcal{L}$, or that they span $\mathcal{L}$. We
call $k$ the rank of $ \mathcal{L}$. The determinant $\text{det}(\mathcal{L})$, of $\mathcal{L}$ is defined by
\begin{align*}
	\text{det}(\mathcal{L}) = | \det(b_1, b_2, \ldots, b_k) |,
\end{align*}
with the $b_i$'s being written as column vectors. This is a positive real number that does not depend on the choice of the basis (see \cite{Cas}, Section 1.2).

Given linearly independent vectors $b_1, b_2, \ldots, b_k $ in $ \mathbb{R}^k$, we refer back to the Gram--Schmidt orthogonalization technique. This method allows us to inductively define vectors $b^*_i$ (with $1 \leq i \leq k$) and real coefficients $\mu_{i,j}$ (for $1 \leq j \leq i \leq k$). Specifically,
\begin{align*}
	b^*_i &= b_i - \sum_{j=1}^{i-1} \mu_{i,j} b^*_j,~~~
	\mu_{i,j} = \dfrac{\langle b_i, b^*_j\rangle }{\langle b^*_j, b^*_j\rangle},
\end{align*}
where \( \langle \cdot , \cdot \rangle \)  denotes the ordinary inner product on \( \mathbb{R}^k \). Notice that \( b^*_i \) is the orthogonal projection of \( b_i \) on the orthogonal complement of the span of \( b_1, \ldots, b_{i-1} \), and that \( \mathbb{R}b_i \) is orthogonal to the span of \( b^*_1, \ldots, b^*_{i-1} \) for \( 1 \leq i \leq k \). It follows that \( b^*_1, b^*_2, \ldots, b^*_k \) is an orthogonal basis of \( \mathbb{R}^k \). 
\begin{definition}
	The basis $b_1, b_2, \ldots, b_n$ for the lattice $\mathcal{L}$ is called reduced if
	\begin{align*}
		\| \mu_{i,j} \| &\leq \frac{1}{2}, \quad \text{for} \quad 1 \leq j < i \leq n,~~
		\text{and}\\
		\|b^*_{i}+\mu_{i,i-1} b^*_{i-1}\|^2 &\geq \frac{3}{4}\|b^*_{i-1}\|^2, \quad \text{for} \quad 1 < i \leq n,
	\end{align*}
	where $ \| \cdot \| $ denotes the ordinary Euclidean length. The constant $ {3}/{4}$ above is arbitrarily chosen, and may be replaced by any fixed real number $ y $ in the interval ${1}/{4} < y < 1$ {\rm(see \cite{LLL}, Section 1)}.
\end{definition}
Let $\mathcal{L}\subseteq\mathbb{R}^k$ be a $k-$dimensional lattice  with reduced basis $b_1,\ldots,b_k$ and denote by $B$ the matrix with columns $b_1,\ldots,b_k$. 
We define
\[
l\left( \mathcal{L},y\right)= \left\{ \begin{array}{c}
	\min_{x\in \mathcal{L}}||x-y|| \quad  ;~~ y\not\in \mathcal{L}\\
	\min_{0\ne x\in \mathcal{L}}||x|| \quad  ;~~ y\in \mathcal{L}
\end{array}
\right.,
\]
where $||\cdot||$ denotes the Euclidean norm on $\mathbb{R}^k$. It is well known that, by applying the
LLL--algorithm, it is possible to give in polynomial time a lower bound for $l\left( \mathcal{L},y\right)$, namely a positive constant $\delta$ such that $l\left(\mathcal{L},y\right)\ge \delta$ holds (see \cite{SMA}, Section V.4).
\begin{lemma}[\cite{SMA}, Section V.4]\label{lem2.5m}
	Let $b_1, \dots, b_k$ be an LLL-reduced basis for a lattice $\mathcal{L}$ and $b_1^*, \dots, b_k^*$ be the corresponding Gram-Schmidt orthogonal basis. Let $y\in\mathbb{R}^k$ and $z=B^{-1}y$.
	\begin{enumerate}[\upshape(i)]
		\item If $y\not \in \mathcal{L}$, let $i_0$ be the largest index such that $z_{i_0}\ne 0$ and put $\lambda:=\{z_{i_0}\}$.
		\item If $y\in \mathcal{L}$, put $\lambda:=1$.
	\end{enumerate}
	Now, define
	\[
	c_1 := \max_{1\le j\le k}\left\{\dfrac{\|b_1\|}{\|b_j^*\|}\right\}.
	\]
	Then, $l(\mathcal{L}, y) \ge  \delta = \lambda\|b_1\|c_1^{-1}$.
\end{lemma}

In our application, we are given real numbers $\eta_0,\eta_1,\ldots,\eta_k$ which are linearly independent over $\mathbb{Q}$ and two positive constants $c_3$ and $c_4$ such that 
\begin{align}\label{2.9m}
	|\eta_0+x_1\eta_1+\cdots +x_k \eta_k|\le c_3 \exp(-c_4 H),
\end{align}
where the integers $x_i$ are bounded as $|x_i|\le X_i$ with $X_i$ given upper bounds for $1\le i\le k$. We write $X_0:=\max\limits_{1\le i\le k}\{X_i\}$. The basic idea in such a situation, from \cite{Weg}, is to approximate the linear form \eqref{2.9m} by an approximation lattice. So, we consider the lattice $\mathcal{L}$ generated by the columns of the matrix
$$ \mathcal{A}=\begin{pmatrix}
	1 & 0 &\ldots& 0 & 0 \\
	0 & 1 &\ldots& 0 & 0 \\
	\vdots & \vdots &\vdots& \vdots & \vdots \\
	0 & 0 &\ldots& 1 & 0 \\
	\lfloor C\eta_1\rfloor & \lfloor C\eta_2\rfloor&\ldots & \lfloor C\eta_{k-1}\rfloor& \lfloor C\eta_{k} \rfloor
\end{pmatrix} ,$$
where $C$ is a large constant usually of the size of about $X_0^k$ . Let us assume that we have an LLL--reduced basis $b_1,\ldots, b_k$ of $\mathcal{L}$ and that we have a lower bound $l\left(\mathcal{L},y\right)\ge \delta$ with $y:=(0,0,\ldots,-\lfloor C\eta_0\rfloor)$. Note that $ \delta$ can be computed by using the results of Lemma \ref{lem2.5m}. Then, with these notations the following result  is Lemma VI.1 in \cite{SMA}.
\begin{lemma}[Lemma VI.1 in \cite{SMA}]\label{lem2.6m}
	Let $S:=\displaystyle\sum_{i=1}^{k-1}X_i^2$ and $T:=\dfrac{1+\sum_{i=1}^{k}X_i}{2}$. If $\delta^2\ge T^2+S$, then inequality \eqref{2.9m} implies that we either have $x_1=x_2=\cdots=x_{k-1}=0$ and $x_k=-\dfrac{\lfloor C\eta_0 \rfloor}{\lfloor C\eta_k \rfloor}$, or
	\[
	H\le \dfrac{1}{c_4}\left(\log(Cc_3)-\log\left(\sqrt{\delta^2-S}-T\right)\right).
	\]
\end{lemma}

Finally, we present an analytic argument which is Lemma 7 in \cite{GL}.  
\begin{lemma}[Lemma 7 in \cite{GL}]\label{Guz} If $ r \geq 1 $, $T > (4r^2)^r$ and $T >  \dfrac{x}{(\log x)^r}$, then $$x < 2^r T (\log T)^r.$$	
\end{lemma}
SageMath 10.6 is used to perform all computations in this work.

\section{The proof of the main result}

The proof of Theorem \ref{thm1.1} follows in two different cases depending on $n$ versus $k$.

\subsection{The case $2\le n\le k+1$}

In this case, we prove the following result.
 
\begin{lemma}
	\label{lem5.1} 
	Let $p$ be a prime and $n\ge 2$, $k\ge 2$, $\mathfrak{a}\ge 1$ be positive integers.
	Then, the Diophantine equation \eqref{eq:main} has no solution whenever $2\le n\le k+1$.
\end{lemma}
\begin{proof}
Assume first that $2\le n\le k$. Then, Equation \eqref{eq:main} can be written as 
	\begin{align}\label{eq:small-n}
		3\cdot 2^{n-2} = (p+1)p^\mathfrak{a}-1.
	\end{align}
The right--hand side is at least $(2+1)\cdot 2-1\ge 5$, so in the left-hand side we must have $n\ge 3$. We get a contradiction since the left--hand side is even and the right--hand side is odd.	
If $n=k+1$, then \eqref{eq:main} becomes 
	\begin{align*}
		3\cdot 2^{k-1}-2 = (p+1)p^\mathfrak{a}-1,
	\end{align*}
and we get the same contradiction namely that the left--hand side is even and the right--hand side is odd.	This completes the proof of Lemma \ref{lem5.1}.
\end{proof}

Throughout the remaining part of the paper, we consider the case $n\ge k+2$.

\subsection{The case $ n\ge k+2$}
\subsubsection{Bounding $n$ in terms of $k$ and $p$}
Since $ n\ge k+2$ and $k\ge 2$, we have $n\ge 4$. We proceed by proving a series of results.

\begin{lemma}\label{lem5.2} 
	Let $p$ be a prime and $n\ge k+2$, $k\ge 2$, $\mathfrak{a}\ge 1$ be positive integers.
	In the Diophantine equation \eqref{eq:main}, we have
	\begin{align*}
		n&<3.6\cdot 10^{15} k^{4}(\log k)^3 (\log p)^3 .
	\end{align*}
\end{lemma}
\begin{proof}
We begin by rewriting Inequality \eqref{lk_b} using \eqref{eq:main} as
\begin{align*}
	\left|(p+1)p^\mathfrak{a}-1 - f_k(\alpha)(2\alpha - 1)\alpha^{n - 1}\right| < \frac{3}{2}, 
\end{align*}
or equivalently $\left|(p+1)p^\mathfrak{a} - f_k(\alpha)(2\alpha - 1)\alpha^{n - 1}\right| < 5/2$.
Dividing both sides by $f_k(\alpha)(2\alpha-1)\alpha^{n-1}$, which is positive because $\alpha>1$, we get
\begin{align}\label{eq:lin-f1}
	\left|(p+1)p^\mathfrak{a}\cdot(2\alpha-1)^{-1}\cdot\alpha^{-(n-1)}\cdot(f_k(\alpha))^{-1}-1\right|&<\dfrac{5}{2f_k(\alpha)(2\alpha-1)\alpha^{n-1}}\nonumber\\
	&<\dfrac{5}{\alpha^{n}},
\end{align}
where in the second inequality, we used relation \eqref{eq2.5}; i.e., $f_k(\alpha)>1/2$ and the relation \eqref{eq2.3}; i.e., $1.5<\alpha<2$, which holds true for all $k\ge 2$. Let 
\begin{align*}
\Gamma_1:=(p+1)p^\mathfrak{a}\cdot(2\alpha-1)^{-1}\cdot\alpha^{-(n-1)}\cdot(f_k(\alpha))^{-1}-1=e^{\Lambda_1}-1.
\end{align*} 
It is clear that $\Lambda_1\ne 0$, otherwise we would have
\begin{align*}
	(p+1)p^\mathfrak{a} &= (2\alpha-1)\alpha^{n-1}f_k(\alpha)\nonumber\\
	&=\dfrac{\alpha-1}{2+(k+1)(\alpha-2)}	(2\alpha-1)\alpha^{n-1}.
\end{align*}
Conjugating the above relation by some automorphism of the Galois group of the splitting field of $g_k (x)$ over $\mathbb{Q}$ which sends $\alpha$ to $\alpha_i$ for some $i>1$ and then taking absolute values, we get 
\begin{align}\label{eq3.3p}
	(p+1)p^\mathfrak{a}
	&=\left|\dfrac{\alpha_i-1}{2+(k+1)(\alpha_i-2)}	(2\alpha_i-1)\alpha_i^{n-1}\right|.
\end{align}
Note that from \eqref{eq3.3p}, we have that $|2+(k+1)(\alpha_i-2)|\ge (k+1)|\alpha_i-2|-2>k-1$, as shown on page 1355 of \cite{Brl}. Hence, the right--hand side of \eqref{eq3.3p} becomes
\begin{align*}
	6 =(2+1)2^1&\le (p+1)p^\mathfrak{a}=\left|\dfrac{\alpha_i-1}{2+(k+1)(\alpha_i-2)}(2\alpha_i-1)	\alpha_i^{n-1}\right| \\
	&< \dfrac{|\alpha_i-1|\cdot|2\alpha_i-1|\cdot |\alpha_i|^{n-1}}{k-1}\le\dfrac{2\cdot 3\cdot 1}{k-1}
	 \le 6,
\end{align*} 
a contradiction.  So, $\Lambda_1\ne 0$. 

Next, we intend to apply Theorem \ref{thm:Matl} to $\Gamma_1$. The algebraic number field containing the following $\gamma_i$'s is $\mathbb{K} := \mathbb{Q}(\alpha)$. We have $D = k$, $t := 3$,
\begin{alignat*}{3}
	\lambda_{1} &:= p,         &\qquad \lambda_{2} &:= \alpha,         &\qquad \lambda_{3} &:= \dfrac{p+1}{(2\alpha-1)f_k(\alpha)},      \\
	b_1         &:= \mathfrak{a},        &\qquad b_2    &:= -(n-1),         &\qquad b_3    &:= 1.
\end{alignat*}
Since $h(\gamma_{1})= \log p$, we take $A_1:=k \log p$. On the other hand, $h(\gamma_{2}) =(\log \alpha)/k <0.7/k$, so we take $A_{2}:=0.7$.
For $A_3$, we first compute 
\begin{align*}
	h(\gamma_{3})&:=h\left(\dfrac{p+1}{(2\alpha-1)f_k(\alpha)}\right)\\
	&\le h(p+1)+ h\left((2\alpha-1)\right)+h\left(f_k(\alpha)\right)\\
	&< \log p +2\log 2+\log \alpha+3\log k\\
	&<11\log p\log k,
\end{align*}
for all $k\ge 2$ and $p\ge 2$. Note that in the second inequality above, we have used \eqref{eqh}. So, we can take $A_3:=11k\log k\log p$. Next, $B \geq \max\{|b_i|:i=1,2,3\}$, and by relation \eqref{a-bound}, we can take $B:=n$. Now, by Theorem \ref{thm:Matl},
\begin{align}\label{eq:lin-g1}
	\log |\Gamma_1| &> -1.4\cdot 30^{6} \cdot 3^{4.5}\cdot k^2 (1+\log k)(1+\log n)\cdot (k\log p)\cdot 0.7\cdot 11k\log k\log p\nonumber\\
	&> - 9.5\cdot 10^{12}k^4 (\log k)^2(\log p)^2\log n.
\end{align}
Comparing \eqref{eq:lin-f1} and \eqref{eq:lin-g1}, we get
\begin{align*}
	n&<2.5\cdot 10^{13}k^4 (\log k)^2(\log p)^2\log n.
\end{align*}
We now apply Lemma \ref{Guz} with  $x:=n$, $r:=1$, $T:=2.5\cdot 10^{13}k^4 (\log k)^2\log p >(4r^2)^r=4$ for all $k\ge 2$ and $p\ge 2$. We get 
\begin{align*}
	n&<2^1\cdot 2.5\cdot 10^{13}k^4 (\log k)^2(\log p)^2\left(\log (2.5\cdot 10^{13}k^4 (\log k)^2(\log p)^2)\right)^1\\
	&=5\cdot 10^{13} k^{4}(\log k)^2 (\log p)^2\left(\log (2.5\cdot 10^{13})+4\log k+2\log\log k+2\log \log p\right)\\
	&<3.6\cdot 10^{15} k^{4}(\log k)^3 (\log p)^3  ,
 \end{align*}
and the proof of Lemma \ref{lem5.2} is complete.		
\end{proof}

We proceed by distinguishing between the values of $p$.

\subsubsection{The sub-case $p < n^{10}$}
Let us assume first that $p < n^{10}$. Then, $\log p < 10\log n$ and Lemma \ref{lem5.2} implies that 
\begin{align*}
	n < 3.6\cdot 10^{18} k^{4}(\log k)^3 (\log n)^3 .
\end{align*}
We apply Lemma \ref{Guz} again with  $x:=n$, $r:=3$, $T:=3.6\cdot 10^{18} k^{4}(\log k)^3 >(4r^2)^r=46656$ for all $k\ge 2$. We obtain 
\begin{align}\label{eq:nk}
	n&<2^3\cdot 3.6\cdot 10^{18} k^{4}(\log k)^3\left(\log (3.6\cdot 10^{18} k^{4}(\log k)^3)\right)^3\nonumber\\
	&=2.88\cdot 10^{19} k^{4}(\log k)^3 \left(\log (3.6\cdot 10^{18})+4\log k+3\log\log k\right)^3\nonumber\\
	&<8\cdot 10^{24} k^{4}(\log k)^6 .
\end{align}

Now, we split the values of $k$.

\medskip

\begin{enumerate}[(a)]
	\item Assume first that $k\le 600$. Then $n< 8\cdot 10^{40}$ via \eqref{eq:nk} and $p < n^{10}<2\cdot 10^{409}$. Since $p = 2^\ell \pm 1$ with $\ell \ge 1$, then $2^\ell \pm 1=p  < 2\cdot 10^{409}$ implies that $\ell \le 1359$.
	
	\medskip
	
	So, we reduce the bound on $n$. To do so, we go back \eqref{eq:lin-f1} and recall that
	\[
	\Gamma_1 := (p+1)p^\mathfrak{a}\cdot(2\alpha-1)^{-1}\cdot\alpha^{-(n-1)}\cdot(f_k(\alpha))^{-1}-1.
	\]
	Assume further that $n\ge 5$, so that we obtain the inequality  
	\[
	\left| e^{\Lambda_1} - 1 \right| = |\Gamma_1| < 0.5,
	\]
	which leads to $ |\log (1+\Gamma_1)| < 1.5|\Gamma_1|$ via \eqref{eq2.5g}. Therefore 
	\[
	\left|  \log \left(\dfrac{p+1}{(2\alpha-1)f_k(\alpha)}\right) + \mathfrak{a} \log p -(n-1) \log \alpha  \right| < \dfrac{7.5}{\alpha^{n}}.
	\]
	For each $k \in [2,600]$ and each prime $p$ of the form $p=2^\ell \pm 1$ with $\ell \in [1,1359]$, we use the LLL-algorithm to obtain a lower bound for the smallest nonzero value of the above linear form, constrained by integer coefficients with absolute values not exceeding $n < 8 \cdot 10^{40}$. In particular, we consider the lattice  
	\[
	\mathcal{A} = \begin{pmatrix} 
		1 & 0 & 0 \\ 
		0 & 1 & 0 \\ 
		\lfloor C\log (1/\alpha)\rfloor & \lfloor C\log p\rfloor & \lfloor C\log \left((p+1)/\left((2\alpha-1)f_k(\alpha) \right)\right) \rfloor
	\end{pmatrix},
	\]
	where we set $C := 1.536\cdot 10^{123}$ and $y := (0,0,0)$. Applying Lemma \ref{lem2.5m}, we obtain 
	\[
	c_1 := 10^{-47} \quad \text{and} \quad \delta := 1.084\cdot 10^{42}. 
	\]
	Using Lemma \ref{lem2.6m}, we conclude that $S =1.92 \cdot 10^{82}$ and $T = 1.21 \cdot 10^{41}$. Since $\delta^2 \geq T^2 + S$, we select $c_3 := 7.5$, $c_4 := \log \alpha$ and establish the bound $n \leq 466$.  This was for $n\ge 5$, but the conclusion $n\le 466$ holds for $n<5$ as well. 
	\medskip
	
	We then use SageMath to find all solutions to \eqref{eq:main} with possible values from the ranges $k\in[2,600]$, $n\in[k+2,466]$, primes $p$ of the form $p=2^\ell\pm 1$ with $\ell\in [1,1359]$ and $a\in[1,n]$. We only obtain the 3 solutions listed in Theorem \ref{thm1.1}.
	
	\medskip
	\begin{remark}
		Although $\ell$ ranges over $[1, 1359]$, only $19$ values of $\ell$ yield a prime of the form $p = 2^{\ell} \pm 1$, namely the $15$ known Mersenne primes and the $5$ known Fermat primes in this range (with $p = 3$ counted only once). Thus the LLL lattice $\mathcal{A}$ is applied to exactly $19$ values of the prime $p$. Even for the largest among them, such as $2^{1279} - 1$ (a $386$-digit number), the computation is feasible because the LLL algorithm operates on 
		$\lfloor C \log p \rfloor$, not on $p$ itself.
	\end{remark}

\medskip
	
\item Finally here, assume $k>600$.	Then the upper bound on $n$ from \eqref{eq:nk} gives
\begin{align*}
	n<8\cdot 10^{24} k^{4}(\log k)^6<2^{k/2}.
\end{align*}
This puts us in position to use inequality \eqref{lk_b1}, which together with \eqref{lk_b} gives
\begin{align*}
	\left|(p+1)p^\mathfrak{a}-3\cdot2^{n-2}\right| &=\left|1+L_n^{(k)} - 3\cdot2^{n-2}\right|\\
	&\le 1+ \left|L_n^{(k)} - f_k(\alpha)(2\alpha - 1)\alpha^{n - 1}\right|+\left|f_k(\alpha)(2\alpha - 1)\alpha^{n - 1} - 3\cdot2^{n-2}\right|\\
	&< 1+\dfrac{3}{2} + 3\cdot 2^{n-2}\cdot \frac{36}{2^{k/2}}\\
	&< 4 +  2^{n}\cdot \frac{27}{2^{k/2}}\\
	&< 2^{n+2}\cdot \frac{27}{2^{k/2}},
\end{align*}
since $n \ge k+2 >k/2$. Now, dividing through by $3\cdot 2^{n-2}$ gives
\begin{align}\label{eq:lin-f2}
	\left|\dfrac{(p+1)}{3}p^\mathfrak{a}\cdot 2^{-(n-2)}-1\right| 	&<  \frac{144}{2^{k/2}}.
\end{align}
Let 
\begin{align*}
	\Gamma_2:=\dfrac{(p+1)}{3}p^\mathfrak{a}\cdot 2^{-(n-2)}-1=e^{\Lambda_2}-1.
\end{align*} 
Note that if $\Lambda_2 = 0$, then
\begin{align*}
	(p+1)p^\mathfrak{a}=3\cdot2^{n-2}.
\end{align*}
However, this would lead to $L_n^{(k)}=(p+1)p^a-1=3\cdot 2^{n-2}-1$, which is false for $n\ge k+1$, thanks to \eqref{eq:3times2}.
 Thus, $\Lambda_2\ne 0$ and we can safely apply Theorem \ref{thm:Matl} to $\Gamma_2$. The algebraic number field containing the following $\gamma_i$'s is $\mathbb{K} := \mathbb{Q}$, so we have $D = 1$, $t := 3$,
\begin{alignat*}{3}
	\lambda_{1} &:= p,         &\qquad \lambda_{2} &:= 2,         &\qquad \lambda_{3} &:= \dfrac{p+1}{3},      \\
	b_1         &:= \mathfrak{a},        &\qquad b_2    &:= -(n-2),         &\qquad b_3    &:= 1.
\end{alignat*}
As before, we take $A_1:= \log p$ and $B:=n$. Moreover, $h(\gamma_{2}) =\log 2 $, so we take $A_{2}:=\log 2$.
For $A_3$, we first compute 
\begin{align*}
	h(\gamma_{3})&:=h\left((p+1)/3\right)
	\le h(p+1)+ h\left(3\right)
	< \log p +\log 2+\log 3<4\log p,
\end{align*}
for all $p\ge 2$. So, we can take $A_3:=4\log p$. Now, by Theorem \ref{thm:Matl},
\begin{align}\label{eq:lin-g2}
	\log |\Gamma_2| &> -1.4\cdot 30^{6} \cdot 3^{4.5}\cdot 1^2 (1+\log 1)(1+\log n)\cdot \log p\cdot \log 2 \cdot 4\log p\nonumber\\
	&> - 6.5\cdot 10^{12} (\log p)^2\log n.
\end{align}
Comparing \eqref{eq:lin-f2} and \eqref{eq:lin-g2} with the fact that $p<n^{10}$, we get
\begin{align}\label{k-b1}
	k&<2\cdot 10^{15}(\log n)^3.
\end{align}	
With this bound of $k$, the bound on $n$ from \eqref{eq:nk} becomes $n<3\cdot 10^{95}(\log n)^{18}$. Applying Lemma \ref{Guz}, we get $n<2\cdot 10^{143}$ and so the bound on $k$ from \eqref{k-b1} becomes $k<8\cdot 10^{22}$. 

\medskip

Thus, we can conclude that whenever $p<n^{10}$ and $k>600$, we have 
\begin{align*}
	n<2\cdot 10^{143}\qquad\text{and} \qquad k<8\cdot 10^{22}.
\end{align*}

Next, we reduce these bounds. To do this, we revisit \eqref{eq:lin-f2} and recall that
	\[
	\Gamma_2 := \dfrac{(p+1)}{3}p^\mathfrak{a}\cdot 2^{-(n-2)}-1.
	\]
	Since $k> 600$, we obtain the inequality  
	\[
	\left| e^{\Lambda_2} - 1 \right| = |\Gamma_2| < 0.5,
	\]
	which leads to $ |\log (1+\Gamma_2)| < 1.5|\Gamma_2|$ via \eqref{eq2.5g}. Therefore,
	\begin{align}\label{eq:gf3}
	\left|  \log \left(\dfrac{p+1}{3}\right) + \mathfrak{a} \log p -(n-2) \log 2  \right| < \dfrac{216}{2^{k/2}}.
\end{align}

Notice that since $2^\ell \pm 1:= p<n^{10}$, then
\begin{align*}
	\ell < \dfrac{10\log (n+1)}{\log 2} < \dfrac{10\log (2\cdot 10^{143}+1)}{\log 2} <4761.
\end{align*}

	Now, for each prime $p$ of the form $p=2^\ell \pm 1$ with $\ell \in [1,4761]$, we use the LLL-algorithm to obtain a lower bound for the smallest nonzero value of the linear form \eqref{eq:gf3}, constrained by integer coefficients with absolute values not exceeding $n < 2 \cdot 10^{143}$. Precisely, we consider the lattice  
\[
\mathcal{A^*} = \begin{pmatrix} 
	1 & 0 & 0 \\ 
	0 & 1 & 0 \\ 
	\lfloor C\log (1/2)\rfloor & \lfloor C\log p\rfloor & \lfloor C\log \left((p+1)/3\right) \rfloor
\end{pmatrix},
\]
where we set $C := 2.4\cdot 10^{430}$ and $y := (0,0,0)$. Applying Lemma \ref{lem2.5m}, we obtain 
\[
c_1 := 10^{-139} \quad \text{and} \quad \delta := 1.403\cdot 10^{144}. 
\]
Using Lemma \ref{lem2.6m}, we conclude that $S =1.2 \cdot 10^{287}$ and $T = 3.01 \cdot 10^{143}$. Since $\delta^2 \geq T^2 + S$, and selecting $c_3 := 216$ and $c_4 := \log 2$, we establish the bound $k/2 \leq 959$. Therefore, 
\begin{align*}
k\le 1918, \qquad \text{and so} \qquad n < 4 \cdot 10^{43},	
\end{align*}
 via \eqref{eq:nk}. This further tells us that
 \begin{align*}
 	\ell < \dfrac{10\log (n+1)}{\log 2} < \dfrac{10\log (4\cdot 10^{43}+1)}{\log 2} <1448.
 \end{align*}
So, for each prime $p$ of the form $p=2^\ell \pm 1$ with $\ell \in [1,1448]$, we repeat this reduction process again to further reduce the bound on $k$. We still consider the latter $\mathcal{A^*}$ where now $C := 1.92\cdot 10^{131}$ and $y := (0,0,0)$. Using Lemma \ref{lem2.5m}, we obtain 
\[
c_1 := 10^{-41} \quad \text{and} \quad \delta := 6.32\cdot 10^{44}. 
\]
On the other hand, Lemma \ref{lem2.6m} gives $S =4.8 \cdot 10^{87}$ and $T = 6.01 \cdot 10^{43}$. Choosing $c_3 := 216$ and $c_4 := \log 2$, we get $k/2 \leq 296$ and so $k \le 592$. This contradicts our working assumption $k > 600$, completing the analysis of the sub-case $p < n^{10}$ and $k > 600$.
\end{enumerate}

\subsubsection{The sub-case $p > n^{10}$}

In this case, we go back to \eqref{eq:main} and note that since $L_n^{(k)} = (p+1)p^\mathfrak{a}-1$ is always odd, then Lemma \ref{lem:1} tells us that 
$$n\equiv r \equiv 1,2 \pmod {k+1}.$$

Now, if $p=2^\ell -1$, then 
$$L_n^{(k)} = (p+1)p^\mathfrak{a}-1 = 2^\ell (2^\ell -1)^\mathfrak{a} -1 \equiv -1 \pmod {2^\ell},$$
 and if $p=2^\ell +1$, then 
 $$L_n^{(k)} = (p+1)p^\mathfrak{a}-1 = (2^\ell +2) (2^\ell +1)^\mathfrak{a} -1 \equiv 2\cdot 1-1\equiv 1 \pmod {2^\ell}.$$
 Therefore, we have
 \begin{align}\label{eq:equiv1}
 L_n^{(k)} \equiv \pm 1 \pmod {2^\ell}.	
 \end{align}
On the other hand, since $n\equiv r \equiv 1,2 \pmod {k+1}$, then the parts (i) and (ii) of Lemma \ref{lem:2adic} tell us that
\begin{align}\label{eq:equiv2}
	L_n^{(k)} \equiv (4m+1)(-1)^m\pmod {2^{k-1}} \qquad \text{or} \qquad L_n^{(k)} \equiv (4m^2+6m+3)(-1)^m\pmod {2^{k}},
\end{align}
respectively, for all $k>600$. Looking at \eqref{eq:equiv1} and \eqref{eq:equiv2}, we obtain that 
\begin{align*}
	 (4m+1)(-1)^m\equiv \pm 1 \pmod {2^{\min\{k-1,\ell\}}} \qquad \text{or} \qquad  
	 (4m^2+6m+3)(-1)^m \equiv \pm 1 \pmod {2^{\min\{k-1,\ell\}}}.
\end{align*}
Thus, $2^{\min\{k-1,\ell\}}$ divides one of $4m$, $4m+2$, $4m^2+6m+2$ or $4m^2+6m+4$. All these numbers are non-zero and are less that $10n^2$.

\medskip

If $\min\{k-1,\ell\}:=\ell$, then $2^\ell \le 10n^2 $. This further implies that $p=2^\ell \pm 1 < 10n^2 +1$, which is not possible because $p>n^{10}$ and $n\ge k+2\ge 4$. Therefore, we always have $\min\{k-1,\ell\}:=k-1$. For this reason,
$2^{k-1} < 10n^2 $ and hence
\begin{align}\label{eq:log-k}
	k <10\log n \qquad \text{and} \qquad \log k < 2\log n.
\end{align}
Also, 
\begin{align*}
	2^{\mathfrak{a}+1} \le (p+1)p^\mathfrak{a}-1 = L_n^{(k)} < 3\cdot 2^{n-2} < 4\cdot 2^{n-2} = 2^n,
\end{align*}
implies that $\mathfrak{a} < n-1$. Therefore,
\begin{align*}
	p^\mathfrak{a} = (2^\ell \pm 1)^\mathfrak{a} = 2^{\mathfrak{a}\ell}\left(1 \pm \dfrac{1}{2^\ell}\right)^\mathfrak{a}.
\end{align*}
When $p$ is a Fermat prime, we have 
$$
2^{(\ell+1)\mathfrak{a}}<(p+1)p^{\mathfrak{a}}=(2^\ell+2)(2^{\ell}+1)^a=2^{(\ell+1)\mathfrak{a}}\left(1+\frac{2}{2^{\ell}}\right)\left(1+\frac{1}{2^{\ell}}\right)^{\mathfrak{a}}.
$$
Note that
\begin{eqnarray*}
\left(1+\frac{1}{2^{\ell}}\right)^{\mathfrak{a}}<e^{\log(1+1/2^{\ell})\mathfrak{a}}<e^{{\mathfrak{a}}/2^{\ell}}<1+\frac{2{\mathfrak{a}}}{2^{\ell}}<1+\frac{2n}{2^{\ell}},
\end{eqnarray*}
where we used the fact that $e^x<1+2x$ for $x\in (0,1/2)$, with $x:={\mathfrak{a}}/2^{\ell}<n/(p-1)\le n/n^{10}=1/n^9<1/2$. Thus, 
$$
2^{(\ell+1)\mathfrak{a}}<(p+1)p^{\mathfrak{a}}<2^{(\ell+1)\mathfrak{a}}\left(1+\frac{2}{2^{\ell}}\right)\left(1+\frac{2n}{2^{\ell}}\right)<2^{(\ell+1)\mathfrak{a}}\left(1+\frac{2+2n}{2^{\ell}}+\frac{4n}{2^{2\ell}}\right)<2^{(\ell+1)\mathfrak{a}}\left(1+\frac{2n+3}{2^{\ell}}\right),
$$
where we used the fact that $2^{\ell}=p-1\ge n^{10}>4n$. In particular, 
\begin{equation}
\label{pand2}
\left|(p+1)p^{\mathfrak{a}}-2^{(\ell+1)\mathfrak{a}}\right|<\frac{2^{(\ell+1)\mathfrak{a}}(2n+3)}{2^{\ell}}.
\end{equation}
This was when $p$ is a Fermat prime. But the same inequality holds when $p$ is a Mersenne prime. Indeed, in this case 
\begin{eqnarray*}
2^{(\ell+1)\mathfrak{a}} & > & (p+1)p^{\mathfrak{a}}=2^{\ell} (2^{\ell}-1)^{\mathfrak{a}}=2^{(\ell+1)\mathfrak{a}}\left(1-\frac{1}{2^{\ell}}\right)^{\mathfrak{a}}\\
& = &
2^{(\ell+1)\mathfrak{a}}e^{\mathfrak{a}\log(1-1/2^{\ell})}>2^{(\ell+1)\mathfrak{a}}e^{-\mathfrak{a}/2^{\ell-1}}\\
& > & 2^{(\ell+1)\mathfrak{a}}\left(1-\frac{\mathfrak{a}}{2^{\ell-1}}\right),
\end{eqnarray*}
which implies that 
$$
\left|(p+1)p^{\mathfrak{a}}-2^{(\ell+1)\mathfrak{a}}\right|<\frac{2^{(\ell+1)\mathfrak{a}} {\mathfrak{a}}}{2^{\ell-1}}<\frac{2^{(\ell+1)\mathfrak{a}} (2n)}{2^{\ell}},
$$
which is slightly stronger than \eqref{pand2}.

\medskip

To proceed, we use \eqref{lk_b} to rewrite \eqref{eq:main} as
\begin{align*}
	\left|f_k(\alpha)(2\alpha - 1)\alpha^{n - 1}- 2^{(\ell+1)\mathfrak{a}}\right| &=	\left|f_k(\alpha)(2\alpha - 1)\alpha^{n - 1}- ((p+1)p^\mathfrak{a}-1)+((p+1)p^\mathfrak{a}-1)- 2^{(\ell+1)\mathfrak{a}}\right|\\
	&\le \left|f_k(\alpha)(2\alpha - 1)\alpha^{n - 1}- ((p+1)p^\mathfrak{a}-1)\right|+\left|((p+1)p^\mathfrak{a}-1)- 2^{(\ell+1)\mathfrak{a}}\right|\\
	&= \left|f_k(\alpha)(2\alpha - 1)\alpha^{n - 1}- L_n^{(k)}\right|+\left|((p+1)p^\mathfrak{a}-1)- 2^{(\ell+1)\mathfrak{a}}\right|\\
	&\le  \left|f_k(\alpha)(2\alpha - 1)\alpha^{n - 1}- L_n^{(k)}\right|+\left|(p+1)p^\mathfrak{a}- 2^{\mathfrak{a}\ell+1}\right| + 1\\
	&<  \dfrac{3}{2} + \dfrac{2^{(\ell+1)\mathfrak{a}} (2n+3)}{2^{\ell}}+1
	= \dfrac{5}{2} +  \dfrac{2^{(\ell+1)\mathfrak{a}} (2n+3)}{2^{\ell}}.
\end{align*}
 Now, dividing through by $2^{(\ell+1)\mathfrak{a}}$ gives
\begin{align}\label{eq:lin-f3}
	\left|f_k(\alpha)(2\alpha - 1)\alpha^{n - 1}\cdot 2^{-(\ell+1)\mathfrak{a}}-1\right| 	&<  \dfrac{5}{2\cdot 2^{(\ell+1)\mathfrak{a}}} + \dfrac{2n+3}{2^{\ell}}<\dfrac{2n+5}{2^{\ell}}\le \dfrac{2n+5}{p-1}.
\end{align}

Let 
\begin{align*}
	\Gamma_3:=f_k(\alpha)(2\alpha - 1)\alpha^{n - 1}\cdot 2^{-(\ell+1)\mathfrak{a}}-1=e^{\Lambda_3}-1.
\end{align*} 
It is clear that $\Lambda_3\ne 0$, otherwise we would have
\begin{align*}
	2^{(\ell+1)\mathfrak{a}} &= (2\alpha-1)\alpha^{n-1}f_k(\alpha)\nonumber\\
	&=\dfrac{\alpha-1}{2+(k+1)(\alpha-2)}	(2\alpha-1)\alpha^{n-1}.
\end{align*}
We have already showed that conjugating the above relation by some automorphism of the Galois group of the splitting field of $g_k (x)$ over $\mathbb{Q}$ which sends $\alpha$ to $\alpha_i$ for some $i>1$ and then taking absolute values, we get that 
\begin{align*}
	2^{21}= 2^{(20+1)\cdot 1}&\le 2^{(\ell+1)\mathfrak{a}}=\left|\dfrac{\alpha_i-1}{2+(k+1)(\alpha_i-2)}(2\alpha_i-1)	\alpha_i^{n-1}\right| \\
	&< \dfrac{|\alpha_i-1|\cdot|2\alpha_i-1|\cdot |\alpha_i|^{n-1}}{k-1}\le\dfrac{2\cdot 3\cdot 1}{k-1}
	\le 6,
\end{align*} 
which is absurd.  In the above we used the fact that $n\ge k+2$, so $n\ge 4$ and $p>n^{10}$, so $2^{\ell}\ge p-1\ge n^{10}\ge 4^{10}=2^{20}$, so $\ell\ge 20$. So, we can apply Theorem \ref{thm:Matl} to $\Gamma_3$. The algebraic number field containing the following $\gamma_i$'s is $\mathbb{K} := \mathbb{Q}(\alpha)$. We have $D = k$, $t := 3$,
\begin{alignat*}{3}
	\lambda_{1} &:= 2,         &\qquad \lambda_{2} &:= \alpha,         &\qquad \lambda_{3} &:= (2\alpha-1)f_k(\alpha),      \\
	b_1         &:= -(\ell+1)\mathfrak{a},        &\qquad b_2    &:= n-1,         &\qquad b_3    &:= 1.
\end{alignat*}
Since $h(\gamma_{1})= \log 2$, we take $A_1:=k \log 2$. Moreover, $h(\gamma_{2}) =(\log \alpha)/k <0.7/k$, so we take $A_{2}:=0.7$.
For $A_3$, we first compute 
\begin{align*}
	h(\gamma_{3})&:=h\left((2\alpha-1)f_k(\alpha)\right)
	\le h\left(2\alpha-1\right)+h\left(f_k(\alpha)\right)\\
	&< 2\log 2+\log \alpha+3\log k
	<6\log k,
\end{align*}
for all $k\ge 2$. So, we can take $A_3:=6 k\log k$. Here, we can take $B:=n^2$. Now, by Theorem \ref{thm:Matl},
\begin{align}\label{eq:lin-g3}
	\log |\Gamma_3| &> -1.4\cdot 30^{6} \cdot 3^{4.5}\cdot k^2 (1+\log k)(1+\log (n^2))\cdot (k\log 2)\cdot 0.7\cdot 6k\log k\nonumber\\
	&> - 4\cdot 10^{12}k^4 (\log k)^2\log n.
\end{align}
Comparing \eqref{eq:lin-f3} and \eqref{eq:lin-g3}, we get
\begin{align*}
	\log (p-1) &< \log(2n+5)+ 4\cdot 10^{12}k^4 (\log k)^2\log n\\
	&< \log(8n)+ 4\cdot 10^{12}k^4 (\log k)^2\log n.
\end{align*}
Therefore, we have 
\begin{align*}
	\log p &<  5\cdot 10^{12}k^4 (\log k)^2\log n.
\end{align*}
Using the bounds on $k$ from \eqref{eq:log-k}, we obtain 
\begin{align*}
	\log p &<  2\cdot 10^{17}(\log n)^7.
\end{align*}

Now, we go back to Lemma \ref{lem5.2} and obtain 
\begin{align*}
	n&<3.6\cdot 10^{15} k^{4}(\log k)^3 (\log p)^3 \\
	&<3.6\cdot 10^{15} (10\log n)^{4}(2\log n)^3 \left(2\cdot 10^{17}(\log n)^7\right)^2\\
	&< 2\cdot 10^{55} (\log n)^{21}.
\end{align*}
We now apply Lemma \ref{Guz} with  $x:=n$, $r:=21$, $T:=10^{55}$ and obtain
\begin{align*}
	n&<2^{21}\cdot2\cdot 10^{55} \left(\log2\cdot10^{55} \right)^{21} < 7 \cdot  10^{105}.
\end{align*}	
Furthermore, we use the bound on $k$ from \eqref{eq:log-k} to obtain 
\begin{align*}
	k < 10\log n < 10\log (7\cdot 10^{105}) < 2438.
\end{align*}
Therefore, since $p>n^{10}$, we always have 
\begin{align}\label{eq:nk2}
n< 7 \cdot  10^{105}\qquad\text{and} \qquad	k < 2438.
\end{align}

We now use these bounds to find an absolute bound on $\log p$ and hence $\ell$. To do this, we revisit \eqref{eq:lin-f3} and note that since $p> n^{10}>4^{10}$ for all $n\ge k+2$, we obtain the inequality  
\[
\left| e^{\Lambda_3} - 1 \right| = |\Gamma_3| < 0.5,
\]
which leads to $ |\log (1+\Gamma_3)| < 1.5|\Gamma_3|$. Thus,
\begin{align*}
	\left|  \log \left(f_k(\alpha)(2\alpha - 1)\right) - (\ell+1)\mathfrak{a} \log 2 +(n-1) \log \alpha  \right| < \dfrac{1.5(2n+5)}{p-1}<\dfrac{12n}{p-1}.
\end{align*}

Now, for each $k \in [2,2438]$, we perform the LLL-algorithm to obtain a lower bound for the smallest nonzero value of the linear form above, constrained by integer coefficients with absolute values not exceeding $n^2 < 5 \cdot 10^{211}$. So, we consider the lattice  
\[
\mathcal{A^{**}} = \begin{pmatrix} 
	1 & 0 & 0 \\ 
	0 & 1 & 0 \\ 
	\lfloor C\log (1/2)\rfloor & \lfloor C\log \alpha \rfloor & \lfloor C\log \left(f_k(\alpha)(2\alpha - 1)\right) \rfloor
\end{pmatrix},
\]
with $C := 3.75\cdot 10^{635}$ and $y := (0,0,0)$. Applying Lemma \ref{lem2.5m}, we obtain 
\[
c_1 := 10^{-203} \quad \text{and} \quad \delta := 8.2241\cdot 10^{212}. 
\]
Using Lemma \ref{lem2.6m}, we get $S =7.5 \cdot 10^{423}$ and $T = 7.51 \cdot 10^{211}$. Since $\delta^2 \geq T^2 + S$, and selecting $c_3 := 12n<8.4\cdot 10^{106}$ and $c_4 := 1$, we obtain $\log (p-1) \leq 1219$. Furthermore, since $p=2^\ell \pm 1$, we get that $\ell \le 1760$.

To conclude, we wrote a program in SageMath to find all primes $p$ of the form $p=2^\ell \pm 1$, with $1\le \ell \le 1760$. We obtained a list of 15 primes corresponding to only 
$$\ell =2,3,5,7,13,17,19,31,61,89,107,127,521,607,1279,$$ 
for Mersenne primes, and a list of 5 primes corresponding to only 
$$\ell = 1,2,4,8,16,$$
for Fermat primes. 

To finish this off, we reduce the bound on $n$ from \eqref{eq:nk2}. Since $\ell \leq 1760$, the prime $p$ belongs to a fixed finite set of $19$ primes. For each such prime $p$ and each $k \in [2, 2438]$, we revisit \eqref{eq:lin-f1} and recall that
\[
\left| e^{\Lambda_1} - 1 \right| = |\Gamma_1| < \dfrac{5}{\alpha^n} < 0.5,
\]
for $n \geq 5$, which gives $|\Lambda_1| < 1.5|\Gamma_1|$ via \eqref{eq2.5g}. Therefore,
\begin{align}\label{eq:gf4}
	\left| \log\left(\frac{p+1}{(2\alpha-1)f_k(\alpha)}\right) 
	+ \mathfrak{a}\log p - (n-1)\log\alpha \right| < \frac{7.5}{\alpha^n}.
\end{align}
The integer coefficients $\mathfrak{a}$ and $n-1$ satisfy $|\mathfrak{a}|, |n-1| < n < 7\cdot 10^{105}$. For each pair $(k, p)$ with $k \in [2, 2438]$ and $p = 2^{\ell}\pm 1$ with $\ell \leq 1760$, we apply the LLL-algorithm to obtain a lower bound for the smallest nonzero value of the linear form \eqref{eq:gf4}. We consider the lattice
\[
\mathcal{B} = \begin{pmatrix}
	1 & 0 & 0 \\
	0 & 1 & 0 \\
	\lfloor C\log p \rfloor & \lfloor C\log\alpha \rfloor & 
	\lfloor C\log\left((p+1)/((2\alpha-1)f_k(\alpha))\right) \rfloor
\end{pmatrix},
\]
with $C := 1.029\cdot 10^{318}$ and $y := (0,0,0)$. Applying Lemma \ref{lem2.5m}, we obtain 
\[
c_1 := 10^{-98} \quad \text{and} \quad \delta := 1.501\cdot 10^{107}. 
\]
Using Lemma \ref{lem2.6m}, we conclude that $S =1.47 \cdot 10^{212}$ and $T = 1.051 \cdot 10^{106}$. Using $c_3 := 7.5$ and $c_4 := \log \alpha$, we get $n \leq 1202$.

Since $n \leq 1202$, the condition $p > n^{10}$ forces $p > 1202^{10} > 6.2 \cdot 10^{30}$. Among the $19$ 
Mersenne and Fermat primes with $\ell \leq 1760$, the only ones satisfying this lower bound are those with 
$\ell \geq 107$, namely 
$$p \in \{2^{107}-1,\ 2^{127}-1,\ 2^{521}-1,\ 2^{607}-1,\ 2^{1279}-1\}.$$
We then use SageMath to find all solutions to \eqref{eq:main} with $k \in [2, 2438]$, $n \in [k+2, 1202]$, primes $p = 2^{\ell} \pm 1$ restricted to these five values, and $\mathfrak{a} \in [1, n]$. We find no solutions in this range.
This completes the proof of Theorem \ref{thm1.1}. \qed

Table \ref{tab:cases} summarizes all cases handled in the proof of Theorem \ref{thm1.1}.
\begin{table}[H]
	\centering
	\renewcommand{\arraystretch}{1.5}
	\begin{tabular}{lp{8cm}}
		\hline
		\textbf{Case} & \textbf{Result} \\
		\hline
		$2 \leq n \leq k+1$ & No solutions - Lemma~\ref{lem5.1} \\[4pt]
		$n \geq k+2,\ p < n^{10},\ k \leq 600$ & LLL gives $n \leq 466$; 
		SageMath gives $3$ solutions \\[4pt]
		$n \geq k+2,\ p < n^{10},\ k > 600$ & Contradiction $k \leq 592$ \\[4pt]
		$n \geq k+2,\ p > n^{10}$ & LLL gives $\ell \leq 1760$, then 
		$n \leq 1202$; SageMath gives no new solutions \\
		\hline
	\end{tabular}
	\caption{Summary of cases in the proof of Theorem~\ref{thm1.1}.}
	\label{tab:cases}
\end{table}

\section*{Acknowledgments} 
The first author thanks the Mathematics division of Stellenbosch University for funding his PhD studies. The second author was partially supported by the 2024 ERC Synergy Grant ``DynAMiCs". The third author worked on this paper during a visit to Stellenbosch University in February-March 2026. He thanks the people of this institution for hospitality and financial support.

\section*{Addresses}

$ ^{1} $ Mathematics Division, Stellenbosch University, Stellenbosch, South Africa.

Email: \url{hbatte91@gmail.com}

Email: \url{fluca@sun.ac.za}
\\
$ ^{2} $ Max Planck Institute for Software Systems, Saarbr\"ucken, Germany.
\\
$ ^{3} $ Department of Applied Mathematics, Naval Postgraduate School, Monterey, CA 93943, USA.

Email: \url{pstanica@nps.edu}

\end{document}